\documentclass[reqno, 12pt]{amsart}
\usepackage{dsfont}
\usepackage{amssymb}
\usepackage{mathrsfs}
\usepackage{enumitem}
\usepackage{comment}
\usepackage[backend=biber, style=alphabetic, url=false]{biblatex}
\usepackage[pdfstartview=FitH, pdfborder={0 0 0}, colorlinks=true, citecolor=blue, linkcolor=blue]{hyperref}
\usepackage{aliascnt}

\theoremstyle{plain}
\newtheorem{theorem}{Theorem}[section]

\newaliascnt{conjecture}{theorem}

\aliascntresetthe{conjecture}

\newaliascnt{corollary}{theorem}

\aliascntresetthe{corollary}

\newaliascnt{lemma}{theorem}
\newtheorem{lemma}[lemma]{Lemma}
\aliascntresetthe{lemma}

\newaliascnt{fact}{theorem}

\aliascntresetthe{fact}

\newaliascnt{claim}{theorem}

\aliascntresetthe{claim}

\newaliascnt{proposition}{theorem}

\aliascntresetthe{proposition}

\theoremstyle{definition}

\newaliascnt{definition}{theorem}

\aliascntresetthe{definition}

\newaliascnt{example}{theorem}

\aliascntresetthe{example}

\theoremstyle{remark}

\newaliascnt{remark}{theorem}
\newtheorem{remark}[remark]{Remark}
\aliascntresetthe{remark}

\numberwithin{equation}{section}

\begin{document}
\title[A conjugate radius bound]{A conjugate radius bound for open three-manifolds}
\author[J. Ge]{Jian Ge}
\address[Ge]{School of Mathematical Sciences, Laboratory of Mathematics and Complex Systems, Beijing Normal University, Beijing 100875, P. R. China.}
\email{jge@bnu.edu.cn}
\thanks{NSFC 12371049 and the Fundamental Research Funds for the Central Universities.}

\author[C. Li]{Chuanhuan Li}
\address[C. Li]{Shanghai Institute for Mathematics and Interdisciplinary Sciences (SIMIS), Shanghai 200433, China \newline
${\quad}$ Research Institute of Intelligent Complex Systems, Fudan University, Shanghai 200433, China}
\email{chli@simis.cn}
\thanks{China Postdoctoral Science Foundation (Grant
No. 2026M793350)}

\date{October 6, 2026}
\begin{abstract}
	\setlength{\emergencystretch}{2em}
	We prove that every complete, connected, noncompact Riemannian three-manifold without boundary with scalar curvature at least $6$ has conjugate radius at most $2\pi/3$. This resolves Shen's conjecture on the strict conjugate-radius bound $\operatorname{conj}(M,g)<\pi$ in dimension three.
\end{abstract}
\maketitle

\section{Introduction}\label{sec:introduction}

Green \cite{Green1963} proved that a closed $n$-dimensional Riemannian manifold with average scalar curvature $s>0$ satisfies
\begin{equation*}
	\operatorname{conj}(M,g)\le\pi\sqrt{n(n-1)/s},
\end{equation*}
with equality precisely for metrics of constant positive sectional curvature. For complete noncompact manifolds with $\operatorname{Scal}_g\ge n(n-1)$, Shen \cite{Shen2001} proposed the strict inequality $\operatorname{conj}(M,g)<\pi$. He established the non-strict bound, in particular, under a Ricci lower bound and subexponential volume growth.

Zhu \cite{Zhu2022} proved the non-strict bound under the additional assumption of nonnegative Ricci curvature. Kwong \cite{Kwong2025} established Green's estimate and rigidity for complete finite-volume manifolds with integrable negative Ricci curvature on the unit tangent bundle, including the Ricci-lower-bound setting discussed in \cite{Gromov2025JAMR}. Ge and Zhang \cite{GeZhang2026} obtained $\operatorname{conj}(M,g)\le\pi$ when the bottom of the Laplacian spectrum is zero; they also proved an estimate for the average scalar curvature on finite-volume manifolds, with constant-curvature rigidity under integrability of the negative part of scalar curvature. In dimension three, Richard \cite{Richard2026} proved the unconditional estimate $\operatorname{inj}(M,g)\le2\pi/3$ for complete noncompact manifolds with $\operatorname{Scal}_g\ge6$. Since $\operatorname{inj}(M,g)\le\operatorname{conj}(M,g)$, that estimate does not imply the corresponding conjugate-radius bound.

For a complete Riemannian manifold $(M,g)$, let $\operatorname{conj}_g(x)$ denote the infimum of the first conjugate times along unit-speed geodesics issuing from $x$, with value $+\infty$ when no such conjugate point exists, and set
\begin{equation*}
	\operatorname{conj}(M,g)=\inf_{x\in M}\operatorname{conj}_g(x).
\end{equation*}
We prove the following estimate.

\begin{theorem}\label{thm:main}
	\setlength{\emergencystretch}{2em}
	Let $(M^3,g)$ be a complete, connected, noncompact smooth Riemannian manifold without boundary. If $\operatorname{Scal}_g\ge6$, then
	\begin{equation*}
		\operatorname{conj}(M,g)\le\frac{2\pi}{3}.
	\end{equation*}
\end{theorem}

To our knowledge, \autoref{thm:main} gives the first general conjugate radius bound strictly below $\pi$ for complete noncompact three-manifolds under the scalar curvature lower bound alone. It therefore resolves Shen's conjecture in dimension three, with the uniform bound $2\pi/3$.

Rescaling gives $\operatorname{conj}(M,g)\le2\pi/(3\sqrt\kappa)$ whenever $\operatorname{Scal}_g\ge6\kappa>0$. The product $S^2(1/\sqrt3)\times\mathbb R$ has scalar curvature $6$ and conjugate radius $\pi/\sqrt3$. The optimal universal constant is not determined here.

The proof combines $\mu$-bubbles with the Schoen--Yau spectral inradius estimate \cite{SchoenYau1983}. Stability produces embedded spheres far from a fixed point, with a positive function satisfying $-\Delta u+Ku\ge(3-o(1))u$. The inradius estimate then forces a component of the inverse image of such a sphere under a nonsingular exponential map to be compact. This component maps diffeomorphically onto the sphere, yielding a contraction that contradicts its nonzero homology class in the punctured manifold. The exponential map need not be injective.

${}$

{\bf Declaration of AI}. The authors provided the main ideas. AI tools were used solely for assistance with writing and language editing. The authors take full responsibility for the content.

\section{Proof}\label{sec:proof}

We use $\Delta=\operatorname{div}\nabla$, write $K$ for Gaussian curvature, and take $H=\operatorname{tr}A$ with respect to the outward unit normal. Homology groups have real coefficients and are ordinary singular homology groups.

To prove \autoref{thm:main}, we may pass to the universal Riemannian cover $\pi:(\widetilde M,\pi^*g)\to(M,g)$. The covering manifold is complete and noncompact and has the same scalar curvature lower bound. Geodesics and Jacobi fields lift under a Riemannian covering, so
\begin{equation*}
	\operatorname{conj}(\widetilde M,\pi^*g)=\operatorname{conj}(M,g).
\end{equation*}
Thus it suffices to prove the theorem for simply connected, and hence oriented, manifolds.

We use the following intrinsic form of the Schoen--Yau radius argument; see \cite[pp.~578--579]{SchoenYau1983}. Its spectral hypothesis is naturally supplied by surface stability, as explained in \autoref{rem:stability-origin}. We include a proof for completeness.

\begin{lemma}\label{lem:inradius}
	Let $(S,h)$ be a smooth Riemannian surface carrying a positive smooth function $u$ such that
	\begin{equation}\label{eq:spectral-inequality}
		-\Delta_hu+K_hu\ge cu,\qquad c>0.
	\end{equation}
	If $\Omega\subset S$ is a compact connected smooth surface domain with nonempty boundary, then
	\begin{equation*}
		d_\Omega(x,\partial\Omega)\le D(c):=\frac{2\pi}{\sqrt{3c}}
	\end{equation*}
	for every $x\in\operatorname{Int}\Omega$, where distance is computed using $h$ and curves in $\Omega$.
\end{lemma}

\begin{proof}
	Set $v=\log u$ and $\bar h=u^2h$. From \eqref{eq:spectral-inequality},
	\begin{equation*}
		K_h-\Delta_hv\ge c+|\nabla v|_h^2,
		\qquad K_{\bar h}=u^{-2}(K_h-\Delta_hv).
	\end{equation*}
	Compactness gives a shortest $\bar h$-curve in $\Omega$ from $x$ to $\partial\Omega$. Stop at its first boundary contact and parametrize it by $h$-arclength $s\in[0,\ell]$. Its interior is a $\bar h$-geodesic. Let $T,N$ be its $h$-unit tangent and normal, and write $v'=\langle\nabla v,T\rangle$ and $v_N=\langle\nabla v,N\rangle$. Since $u^{-1}N$ is a parallel $\bar h$-unit normal and $d\bar s=u\,ds$, its index form satisfies
	\begin{align*}
		0\le I_{\bar h}(\varphi N,\varphi N)
		&=\int_0^\ell u\bigl((\varphi'+v'\varphi)^2-(K_h-\Delta_hv)\varphi^2\bigr)\,ds\\
		&\le\int_0^\ell u\bigl((\varphi')^2+2v'\varphi\varphi'-c\varphi^2-v_N^2\varphi^2\bigr)\,ds
	\end{align*}
	for $\varphi\in C_c^\infty(0,\ell)$. Substituting $\varphi=u^{-1/2}\psi$ and completing the square gives
	\begin{align*}
		0&\le\int_0^\ell\left(\frac43(\psi')^2-c\psi^2
		-\frac34\left(v'\psi-\frac23\psi'\right)^2-v_N^2\psi^2\right)\,ds\\
		&\le\int_0^\ell\left(\frac43(\psi')^2-c\psi^2\right)\,ds.
	\end{align*}
	By density in $H_0^1(0,\ell)$, take $\psi(s)=\sin(\pi s/\ell)$. Then $\ell\le2\pi/\sqrt{3c}$, and $d_\Omega(x,\partial\Omega)\le\ell$.
\end{proof}

\begin{remark}\label{rem:stability-origin}
	Stable minimal surfaces provide a geometric source of the spectral positivity in \autoref{lem:inradius}, as in Gromov--Lawson \cite[Section~10, equation~(10.17)]{GromovLawson1983}. Indeed, if $\Sigma$ is a two-sided stable minimal surface in a Riemannian three-manifold with $\operatorname{Scal}_g\ge\sigma>0$, stability and the Gauss equation give
	\begin{align*}
		\int_\Sigma\bigl(|\nabla\varphi|^2+K_\Sigma\varphi^2\bigr)\,dA
		&\ge\frac12\int_\Sigma\bigl(\operatorname{Scal}_g+|A|^2\bigr)\varphi^2\,dA\\
		&\ge\frac{\sigma}{2}\int_\Sigma\varphi^2\,dA
	\end{align*}
	for every $\varphi\in C_c^\infty(\operatorname{Int}\Sigma)$. Thus $-\Delta_\Sigma+K_\Sigma\ge\sigma/2$ in the quadratic-form sense. If $\Sigma$ is closed and connected, its positive first eigenfunction satisfies \eqref{eq:spectral-inequality} with $c=\sigma/2$.

	The Schoen--Yau radius argument converts this spectral positivity into an intrinsic bound on the distance to the boundary. For a stable $\mu$-bubble with $H=\mu$ in the same ambient curvature range, the corresponding stability calculation gives $-\Delta_\Sigma+K_\Sigma\ge\sigma/2-\varepsilon$ in the quadratic-form sense whenever $\frac34\mu^2-|\nabla\mu|_g\ge-\varepsilon$. In \autoref{lem:remote-spheres}, $\sigma=6$ and $\varepsilon\to0$, so positive first eigenfunctions on the resulting closed components supply the required spectral inequality. Pulling back the metric and positive function by a nonsingular exponential map preserves this inequality, so \autoref{lem:inradius} also applies to compact domains in the inverse image of the surface.
\end{remark}

\begin{lemma}\label{lem:contraction}
	Let $(M^3,g)$ be complete, and let $\Sigma\subset M$ be a smooth closed embedded two-sphere carrying a positive smooth function satisfying $-\Delta_\Sigma u+K_\Sigma u\ge cu$ with $c>0$. If $x\in\Sigma$ and
	\begin{equation*}
		D(c)<a<\operatorname{conj}_g(x),
	\end{equation*}
	then the inclusion of $\Sigma$ is null-homotopic in $B_g(x,a)$. In particular, $[\Sigma]=0$ in $H_2(M\setminus\{p\};\mathbb R)$ whenever $d_g(x,p)>a$.
\end{lemma}

\begin{proof}
	The map $\exp_x:B_a(0)\subset T_xM\to M$ has nonsingular differential, so $\widehat g=(\exp_x)^*g$ is a smooth Riemannian metric and $\exp_x$ is a local isometry. The relatively closed embedded surface $P=\exp_x^{-1}(\Sigma)$ contains $0$, and $\widehat u=u\circ\exp_x|_P$ satisfies the same spectral inequality on $P$.

	Choose a regular value $r\in(D(c),a)$ of the radial norm $r_x(v)=|v|_{g_x}$ restricted to $P\setminus\{0\}$. Let $\Omega$ be the component of $P\cap\overline B_r(0)$ containing $0$. It is a compact smooth surface, possibly with boundary. If the boundary is nonempty, it lies in $\{r_x=r\}$. Gauss's lemma gives $|dr_x|_{\widehat g}=1$ away from $0$, so every curve in $\Omega$ from $0$ to its boundary has length at least $r$. By \autoref{lem:inradius},
	\begin{equation*}
		r\le d_\Omega(0,\partial\Omega)\le D(c)<r,
	\end{equation*}
	a contradiction. Thus $\partial\Omega=\varnothing$, and regularity of $r$ implies $\Omega\subset B_r(0)$.

	The local diffeomorphism $\exp_x|_\Omega:\Omega\to\Sigma$ has compact domain and open and closed image, hence is a finite covering onto $\Sigma$. Since $\Omega$ is connected and $\Sigma$ is simply connected, it is a diffeomorphism. Let $F:\Sigma\to\Omega$ be its inverse. The homotopy
	\begin{equation*}
		H(t,y)=\exp_x(tF(y)),\qquad (t,y)\in[0,1]\times\Sigma,
	\end{equation*}
	joins the constant map at $x$ to the inclusion and satisfies
	\begin{equation*}
		d_g(x,H(t,y))\le t|F(y)|_{g_x}<a.
	\end{equation*}
	The homology conclusion follows whenever this ball avoids $p$.
\end{proof}

\begin{lemma}\label{lem:remote-spheres}
	Let $(M^3,g)$ be complete, connected, oriented, noncompact, and without boundary, with $\operatorname{Scal}_g\ge6$. For each $p\in M$, there are arbitrarily large parameters $L$ and smooth closed embedded two-spheres $\Sigma_L$ such that
	\begin{equation}\label{eq:remote-sphere-data}
		d_g(p,\Sigma_L)>L-1,\qquad
		[\Sigma_L]\ne0\quad\text{in }H_2(M\setminus\{p\};\mathbb R),
	\end{equation}
	and a positive smooth function $u_L$ on $\Sigma_L$ satisfying
	\begin{equation}\label{eq:remote-spectral-data}
		-\Delta_{\Sigma_L}u_L+K_{\Sigma_L}u_L\ge c_Lu_L,
		\qquad c_L=3-\frac{4\pi^2}{3L^2}.
	\end{equation}
\end{lemma}

\begin{proof}
	Choose a smooth approximation to distance \cite[Proposition~2.1]{GreeneWu1979} satisfying
	\begin{equation*}
		|\rho-d_g(p,\cdot)|<1,\qquad |\nabla\rho|\le2.
	\end{equation*}
	Its closed sublevels are compact by completeness. Sard's theorem permits arbitrarily large $L$ for which $L,2L,3L$ are regular values. On the compact band $N_L=\{L\le\rho\le3L\}$, put
	\begin{equation*}
		s_L=\frac{\rho-2L}{2},\qquad
		\mu_L=-\frac{4\pi}{3L}\tan\left(\frac{\pi s_L}{L}\right).
	\end{equation*}
	The function $\mu_L$ tends to $+\infty$ at the inner boundary and to $-\infty$ at the outer boundary. Since $|\nabla s_L|\le1$, direct differentiation gives
	\begin{equation}\label{eq:potential-bound}
		\frac34\mu_L^2-|\nabla\mu_L|
		\ge-\frac{4\pi^2}{3L^2}.
	\end{equation}

	Set $E_0=N_L\cap\{\rho<2L\}$ and minimize
	\begin{equation*}
		F(E)=P(E;N_L^\circ)
		-\int_{N_L}(\chi_E-\chi_{E_0})\mu_L\,dV_g
	\end{equation*}
	among finite-perimeter sets in $\mathscr A_L=\{E:E\mathbin\triangle E_0\Subset N_L^\circ\}$. We detail the barrier argument in \cite{ChodoshLi2024}. Write $P(E)=P(E;N_L^\circ)$ and understand set conditions up to null sets.

	Near the compact regular boundary levels, the fields $\eta_{\rm in}=\nabla\rho/|\nabla\rho|$ and $\eta_{\rm out}=-\nabla\rho/|\nabla\rho|$ are smooth with bounded divergence. Since $\mu_L\to+\infty$ at the inner boundary and $\mu_L\to-\infty$ at the outer boundary, choose $0<\delta=\delta_L<L$, independent of $E$, so that the collars
	\begin{equation*}
		C_{\rm in}=\{L\le\rho<L+\delta\},\qquad
		C_{\rm out}=\{3L-\delta<\rho\le3L\}
	\end{equation*}
	satisfy $\mu_L-\operatorname{div}\eta_{\rm in}\ge1$ and $\mu_L+\operatorname{div}\eta_{\rm out}\le-1$ on $C_{\rm in}\cap N_L^\circ$ and $C_{\rm out}\cap N_L^\circ$, respectively.

	For either collar $C$, let $B=\overline C\cap\partial N_L$ and let $\eta$ be its field. If a finite-perimeter set $A$ fills a neighborhood of $B$, Gauss--Green gives
	\begin{align*}
		\int_C\operatorname{div}\eta\,dV_g&=P(C)-\operatorname{Area}(B),\\
		\int_{A\cap C}\operatorname{div}\eta\,dV_g&\le P(A\cap C)-\operatorname{Area}(B).
	\end{align*}
	Here $\eta$ points outward from $C$ at its interior interface and inward to $N_L$ along $B$. Subtraction and perimeter submodularity give
	\begin{equation*}
		P(A\cup C)-P(A)\le P(C)-P(A\cap C)
		\le\int_{C\setminus A}\operatorname{div}\eta\,dV_g.
	\end{equation*}
	Apply this to $A=E$ at the inner collar and $A=N_L\setminus E$ at the outer collar. Complementation preserves relative perimeter, so for every $E\in\mathscr A_L$,
	\begin{align*}
		F(E\cup C_{\rm in})-F(E)
		&\le\int_{C_{\rm in}\setminus E}(\operatorname{div}\eta_{\rm in}-\mu_L)\,dV_g\le0,\\
		F(E\setminus C_{\rm out})-F(E)
		&\le\int_{E\cap C_{\rm out}}(\operatorname{div}\eta_{\rm out}+\mu_L)\,dV_g\le0.
	\end{align*}
	These modifications have compact support in $N_L^\circ$, so the integrals are finite and $T(E)=(E\cup C_{\rm in})\setminus C_{\rm out}$ is admissible. Let $\mathscr A_{L,\delta}$ be the subclass with $C_{\rm in}\subset E$ and $E\cap C_{\rm out}=\varnothing$. Since $F(T(E))\le F(E)$,
	\begin{equation*}
		\inf_{\mathscr A_L}F=\inf_{\mathscr A_{L,\delta}}F.
	\end{equation*}

	The sign of $\mu_L$ gives
	\begin{equation*}
		F(E)=P(E)+\int_{E\mathbin\triangle E_0}|\mu_L|\,dV_g\ge P(E).
	\end{equation*}
	For $E\in\mathscr A_{L,\delta}$, the difference $E\mathbin\triangle E_0$ lies in the fixed compact set $\{L+\delta\le\rho\le3L-\delta\}$, where $\mu_L$ is bounded. BV compactness, $L^1$ continuity of the volume term, and lower semicontinuity of perimeter give a minimizer $E_*$. Equality of the infima makes $E_*$ a minimizer over $\mathscr A_L$, allowing all compactly supported interior variations. Three-dimensional interior regularity gives a smooth boundary in the compact core. No connectedness of $N_L$ is needed. Gluing $E_*$ to $\{\rho\le L\}$ gives a compact smooth region $W_L$ with
	\begin{equation*}
		\{\rho\le L\}\subset W_L\subset\{\rho<3L\},
		\qquad\partial W_L\subset\{L<\rho<3L\}.
	\end{equation*}

	On every component $S$ of $\partial W_L$, substitute $u=1$ and $h=\mu_{L}$ into the first variation formula and the exact second variation identity of Section 3.2 in \cite{ChodoshLi2024}, it follows that $H=\mu_L$ and
	\begin{equation*}
		\int_S\left(|\nabla_{S}\varphi|^2-\left(\frac{1}{2}{\rm Scal}_{g}-K_{S}+\frac{1}{2}(|A|^2+H^{2})+\nu(\mu_L)\right)\varphi^2\right)\,dA\ge0
	\end{equation*}
	for every smooth $\varphi$.  $|A|^2\ge H^2/2$ and \eqref{eq:potential-bound} yield
	\begin{align*}
		\quad\frac12\operatorname{Scal}_g-K_S+&\frac12(|A|^2+H^2)+\nu(\mu_L)\\
        &\ge3-K_S+\frac34\mu_L^2-|\nabla_{M}\mu_L|\ge c_L-K_S.
	\end{align*}
	Consequently,
	\begin{equation*}
		\int_S(|\nabla_{S}\varphi|^2+K_S\varphi^2)\,dA
		\ge c_L\int_S\varphi^2\,dA.
	\end{equation*}
	For $L>2\pi/3$, testing with $\varphi=1$ gives $2\pi\chi(S)\ge c_L\operatorname{Area}(S)>0$. The surface $S$ is orientable, being a boundary in an oriented manifold, and hence is a two-sphere. Its positive first eigenfunction for $-\Delta_S+K_S$ satisfies \eqref{eq:remote-spectral-data}.

	For large $L$, a small normal ball $B_\varepsilon(p)$ lies in $\operatorname{Int}W_L$. The oriented boundary of $W_L\setminus B_\varepsilon(p)$ gives
	\begin{equation*}
		[\partial W_L]=[\partial B_\varepsilon(p)]
		\quad\text{in }H_2(M\setminus\{p\};\mathbb R).
	\end{equation*}
	Since $M$ is connected and noncompact, $H_3(M;\mathbb R)=0$. Excision and the exact sequence of the pair give \cite[Chapter~3]{Hatcher2002}
	\begin{equation*}
		0\longrightarrow H_3(M,M\setminus\{p\};\mathbb R)
		\cong\mathbb R\xrightarrow{\ \partial\ }H_2(M\setminus\{p\};\mathbb R).
	\end{equation*}
	The local orientation generator maps to $[\partial B_\varepsilon(p)]$, which is therefore nonzero. At least one of the finitely many components of $\partial W_L$, denoted by $\Sigma_L$, has nonzero class. Every point on that boundary has distance greater than $L-1$ from $p$, proving \eqref{eq:remote-sphere-data}.
\end{proof}

\begin{proof}[Proof of \autoref{thm:main}]
	By the reduction above, we may assume that $M$ is simply connected and oriented. If $\operatorname{conj}(M,g)>2\pi/3$, choose
	\begin{equation*}
		\frac{2\pi}{3}<a<\operatorname{conj}(M,g)
	\end{equation*}
	and fix $p\in M$. Since $c_L\to3$, we have $D(c_L)\to2\pi/3$. Take $L$ in \autoref{lem:remote-spheres} so large that $D(c_L)<a<L-1$. For any $x\in\Sigma_L$,
	\begin{equation*}
		D(c_L)<a<\operatorname{conj}_g(x),\qquad d_g(x,p)>a.
	\end{equation*}
	By \autoref{lem:contraction}, $[\Sigma_L]=0$ in $H_2(M\setminus\{p\};\mathbb R)$, contradicting \eqref{eq:remote-sphere-data}.
\end{proof}

\begingroup
\setlength{\emergencystretch}{2em}
\printbibliography
\endgroup
\end{document}